\documentclass[12pt]{amsart}

\usepackage{amssymb}
\usepackage{enumitem}
\usepackage{graphicx}

\makeatletter
\@namedef{subjclassname@2020}{%
	\textup{2020} Mathematics Subject Classification}
\makeatother
\usepackage[T1]{fontenc}

\newtheorem{thm}{Theorem}[section]
\newtheorem{lem}[thm]{Lemma}
\newtheorem{prop}[thm]{Proposition}
\newtheorem{coro}[thm]{Corollary}

\theoremstyle{definition}

\newtheorem{rem}[thm]{Remark}

\numberwithin{equation}{section}

\newcommand{\calf}{\mathcal{F}}

\newcommand{\bbn}{\mathbb{N}}
\newcommand{\bbz}{\mathbb{Z}}
\newcommand{\eps}{\varepsilon}
\DeclareMathOperator{\ess}{ess}

\begin{document}
	
	\baselineskip=17pt
	\title[A dynamical approach to nonhomogeneous spectra]{A dynamical approach to nonhomogeneous spectra}	
	\author[J. Li]{Jian Li}
	\address{Department of Mathematics, Shantou University, Shantou 515063, Guangdong, China}
	\email{lijian09@mail.ustc.edu.cn}
	\urladdr{https://orcid.org/0000-0002-8724-3050}
	\author[X. Liang]{Xianjuan Liang}
	\address{Department of Mathematics, Shantou University, Shantou 515063, Guangdong, China}
	\email{liangxianjuan@outlook.com}
	\urladdr{https://orcid.org/0000-0003-1970-3809}	
	\thanks{X. Liang is the corresponding author.}
	\date{}

\begin{abstract}
		Let $\alpha>0$ and $0<\gamma<1$. Define $g_{\alpha,\gamma}\colon \bbn\to\bbn_0$ by $g_{\alpha,\gamma}(n)=\lfloor n\alpha +\gamma\rfloor$, where $\lfloor x \rfloor$ is the 
		largest integer less than or equal to $x$.
		The set $g_{\alpha,\gamma}(\bbn)=\{g_{\alpha,\gamma}(n)\colon n\in\bbn\}$ 
		is called the $\gamma$-nonhomogeneous spectrum of $\alpha$.
		By extension, the functions $g_{\alpha,\gamma}$ are referred to as spectra.
		In 1996, Bergelson, Hindman and Kra showed that the functions $g_{\alpha,\gamma}$ preserve some largeness of subsets of $\bbn$, 
		that is, if a subset $A$ of $\bbn$ is an IP-set, a central set, an IP$^*$-set, or a central$^*$-set,
		then $g_{\alpha,\gamma}(A)$ is the corresponding object for all $\alpha>0$ and $0<\gamma<1$. 
		In 2012, Hindman and Johnson extended this result to include
		several other notions of largeness: C-sets, J-sets, strongly central sets, and
		piecewise syndetic sets.
		We adopt a dynamical approach to this issue and build a correspondence between the preservation of spectra and the lift property of suspension. As an application, we give a unified proof of some known results and also obtain some new results.  
\end{abstract}
	
	
	
	\maketitle
	
	\section{Introduction}
	Throughout this paper, $\bbn$, $\bbn_0$ and $\bbz$ denote the sets of positive integers, non-negative integers and integers, respectively.
	An interesting elementary result in number theory is as follows:
	for two positive real numbers $\alpha$ and $\beta$, $\{\lfloor n\alpha \rfloor\colon n\in\bbn\}$
	and $\{\lfloor n\beta\rfloor \colon n\in\bbn\}$
	are complementary sets in $\bbn$ if and only if $\alpha$ and $\beta$ are irrational and $\frac{1}{\alpha}+\frac{1}{\beta}=1$. 
	This result is frequently referred to as Beatty’s Theorem,
	as it was posed in 1926 by Samuel Beatty as a problem in the American Mathematical Monthly~{\cite{Beatty1926}}, but it may be due originally to John William Strutt (Lord
	Rayleigh)~{\cite{Strutt1877}}. 
	In \cite{Skolem1957-2}, Skolem introduced the more general sets 
	$\{\lfloor n\alpha+\gamma \rfloor \colon n\in\bbn\}$, determining when two such sets can be disjoint. In {\cite{Graham1978}}, Graham, Lin and Lin called the set $\{\lfloor n\alpha+\gamma \rfloor \colon n\in\bbn\}$ as the \emph{$\gamma$-nonhomogeneous spectrum of $\alpha$}.
	Normally, the parameter $\gamma$ ranges over all real numbers. For technical reasons,
	here we will restrict to $0<\gamma<1$.
	In {\cite{Hindman2012}}, by extension Hindman and Johnson referred to the function $g_{\alpha,\gamma}\colon \bbn\to\bbn_0$ defined by $g_{\alpha,\gamma}(n)=\lfloor n\alpha+\gamma\rfloor$  as a \emph{nonhomogeneous spectrum}.
	As in this paper we only consider subsets of positive integers, for convenience we regard $g_{\alpha,\gamma}$ as a function from $\bbn$ to $\bbn$, that is, we ignore the zero in the image of a set.
	
	In {\cite{Bergelson1996}}, Bergelson, Hindman and Kra used elementary, dynamical and algebraic approaches to study some largeness of subsets of $\bbn$ which are preserved by nonhomogeneous spectra. To be more precisely, they showed that 
	if a subset $A$ of $\bbn$ is an IP-set, a central set, a $\Delta$-set, an IP$^*$-set, a central$^*$-set, or  a $\Delta^*$-set,
	then $g_{\alpha,\gamma}(A)$ is the corresponding object. 
	They first proved the results for IP-sets, $\Delta$-sets, IP$^*$-sets and $\Delta^*$-sets by an elementary approach, for central sets by a dynamical approach and (re)proved all the results by an algebraic approach.
	An interesting consequence of this kind of results 
	provides explicit nontrivial examples of sets with these largeness properties. 
	In {\cite{Hindman2012}}, Hindman and Johnson further developed the algebraic approach and extended the results to include
	several other notions of largeness: C-sets, J-sets, strongly central sets, and
	piecewise syndetic sets. 
	
	In this paper, we will develop a modification of the dynamical approach in {\cite{Bergelson1996}} and build a correspondence between the preservation of spectra  and the lift property of suspension. As an application, we give a unified proof of some known results and also obtain some new.
	
	To state our result, we first need some preparations.
	By a \emph{(topological) dynamical system}, we mean a pair $(X,T)$,
	where $X$ is a compact metric space and $T\colon X\to X$ is continuous.
	If $T$ is a homeomorphism, we say the dynamical system $(X,T)$ is \emph{invertible}. 
	If $Y$ is a nonempty closed subset of $X$ with $T(Y)\subset Y$, then $(Y,T|_{Y})$ is called a \emph{subsystem}  of $(X,T)$. 
	A dynamical system $(X,T)$ is called \emph{minimal} if it contains no proper subsystems. Each point
	belonging to some minimal subsystem of $(X,T)$ is called a \emph{minimal point}. 
	By  Zorn’s Lemma, 
	every topological dynamical system has a minimal subsystem. 
	Points $x$ and $y$ of $X$ are \emph{proximal} if and only if there is an infinite sequence $\{n_i\}_{i=1}^\infty$ in $\bbn$ such
	that $\lim\limits_{i\to\infty} T^{n_i} x=\lim\limits_{i\to \infty} T^{n_i} y$.  
	For a point $x\in X$ and a subset $U$ of $X$, let 
	\[
	N(x,U)=\{n\in\bbn\colon T^n x\in U\}.
	\]
	A point $x$ of $X$ is called a \emph{recurrent point} of dynamical system $(X,T)$ if for any neighborhood $U$ of $x$, the recurrent time set $N(x,U)\neq \emptyset$. 
	It is easy to show that $x$ is  a recurrent point of $(X,T)$ if and only if for any neighborhood $U$ of $x$,  $N(x,U)$ is infinite. 
	We denote $R(X,T)$ be the collection of all recurrent points of dynamical system $(X,T)$. 
	It is not difficult to verify that $T(R(X,T))=R(X,T)$. 
	
	In {\cite{F81}} Furstenberg introduced the concept of central set in $\bbn$
	using notions from topological dynamics. 
	A subset $A$ of $\bbn$ is called a \emph{central set} 
	if there exist a dynamical system $(X,T)$, $x,y\in X$ and a neighborhood $U$ of $y$ such that $y$ is a minimal point, $x,y$ are proximal and $N(x,U)\subset A$. 
	To get a generalization of the central set, in {\cite{Li2012}}  
	the first author of this paper introduced the concepts of strong
	$\calf$-proximality and essential $\calf$-set. This will be explained in Section 3.
	
	We say that a non-empty collection $\calf$ of non-empty subsets of $\bbn$ 
	is a \emph{Furstenberg family} if it is 
	hereditary upwards, that is, $A\in\calf$ and $A\subset B\subset\bbn$ imply $B\in\calf$.
	For $n\in\bbz$ and $A,B\subset \bbn$, let $n+A=\{n+m\colon m\in A\}\cap \bbn$ and $B-A=\{m-n\colon m\in B, n\in A\}\cap \bbn$. 
	We say that a Furstenberg family $\calf$ is \emph{translation $+$ invariant} if for any $n\in\bbn$ and $A\in\calf$, $n+A\in\calf$; and has \emph{the Ramsey property} if for any $A\cup B\in \calf$, either $A\in\calf$ or $B\in\calf$.
	Let $(X,T)$ be a dynamical system and $x,y\in X$.
	We say that $x$ is \emph{$\calf$-strongly proximal to $y$} if for any neighborhood $U$ of $y$,  $N((x,y),U\times U)\in \calf$, where $(x,y)$ and $U\times U$ are considered in the product system $(X\times X,T\times T)$, where $T\times T(x,y)=(Tx,Ty)$ for any $x,y\in X$.
	A subset $A\subset \bbn$ is called an \emph{essential $\calf$-set}, 
	if there exist a dynamical system $(X,T)$, $x,y\in X$ and  a neighborhood $U$ of $y$ such that  $x$ is $\calf$-strongly proximal to $y$ and $N\big((x,y),U\times U\big)\subset A$.
	We also say $A$ is an essential $\calf$-set via $(X,T,x,y,U)$. 
	The collection of all essential $\calf$-sets is denoted by $\ess(\calf)$. 
	If $\calf$ is a Furstenberg family, then $\mathbb{N}\in \calf$ and $\ess(\calf)\subset \calf$. At this time, $\bbn$ is an essential $\calf$-set via $(X,T,x,x,U)$, where $T$ is the identity map on $X$. It means that $\bbn \in \ess(\calf)\neq\emptyset$. 
	By the definition of essential $\calf$-set, we know that  $\ess(\calf)$ is hereditary upwards, so $\ess(\calf)$ is also a Furstenberg family.  
	
	Let $(X,T)$ be an invertible dynamical system. 
	In the product space $X \times [0, 1]$, 
	identify $(x, 1)$ and $(T x, 0)$ for all $x \in X$, let
	$Y = X \times [0, 1)$ have the quotient topology resulting from this identification. 
	Since $X \times [0, 1]$ is a compact metrizable space, 
	then $Y$ is a compact metrizable space. 
	And for each $s \in \mathbb{R}$ the function $F_s : Y\to  Y$
	defined by $F_s(x,t)=\big(T^{\lfloor s+t\rfloor}x,s+t-\lfloor s+t\rfloor\big)$,  $F_s$ is continuous. 
	Then $(Y,F_s)$ is a dynamical system and  is called a \emph{suspension} of $(X,T)$. 
	It is not hard to prove that for any $n\in \bbz$, $s\in\mathbb{R}$ and $(x,t)\in Y$, $F_s^n(x,t)=\big(T^{\lfloor n s+t\rfloor}x,n s+t-\lfloor n s+t\rfloor\big)$. 
	
	One of the main results of this paper is the following correspondence principle between the preservation of spectra and the lift property of suspension. 
	
	\begin{thm}\label{thm:main-result-1}
		Let $\calf$ be a translation $+$ invariant Furstenberg family and $\alpha>0$.
		Assume that the collection  $\ess(\calf)$  has the Ramsey property. 
		Then the following two statements are equivalent:
		\begin{enumerate}
			\item for any $A\in \ess(\calf)$ and $0<\gamma<1$, $g_{\alpha,\gamma}(A)\in \ess(\calf)$;
			\item for any invertible dynamical system $(X,T)$ and any $x,y\in X$ with $x$ $\calf$-strongly proximal to $y$, there exists $0<\gamma_0<1$ such that $(x,\gamma_0)$ is $\calf$-strongly proximal to $(y,\gamma_0)$ for the suspension $(Y,F_{\alpha^{-1}})$ of $(X,T)$.
		\end{enumerate}
	\end{thm}
	
	We will show that the statement (1) in Theorem~\ref{thm:main-result-1} holds under some conditions only depending on the family $\calf$, which leads to the following combinatorial result.
	
	\begin{thm}\label{thm:main-result-2}
		Let $\calf$ be a translation $+$ invariant Furstenberg family with the Ramsey property and $\alpha>0$. 
		If  for any  $A\in\calf$ and $0<\gamma<1$, $g_{\alpha,\gamma}(A)\in \calf$. 
		Then  
		for any $A\in\ess(\calf)$ and $0<\gamma<1$,
		$g_{\alpha,\gamma}(A)\in \ess(\calf)$.
	\end{thm}
	
	We have the following consequence of Theorem~\ref{thm:main-result-2}, which gives a unified proof on the preservation of some largeness of subsets of $\bbn$. 
	
	\begin{coro}\label{cor:IP-D-quasi-central-spectra}
		If a subset $A$ of $\bbn$ is an IP-set, a C-set, a D-set or quasi-central set,
		then so is $g_{\alpha,\gamma}(A)$ for all $\alpha>0$ and $0<\gamma<1$.
	\end{coro}
	
	The paper is organized as follows: 
	In Section 2, we introduce some notations and establish some conventions to be used in the paper. We introduce some largeness of subsets of $\bbn$ and state the known dynamical or algebraic characterizations of them. 
	In Section 3, we build a correspondence between the preservation of spectra and the lift property of suspension  to prove 
	Theorem \ref{thm:main-result-1} and Theorem \ref{thm:main-result-2}. As an application, we give a unified proof of some known results and also obtain some new results.

\section{Preliminaries}
	
	We first introduce some largeness of subsets of $\bbn$. 
	As an application of the main results of this paper, we will obtain some properties about these sets in the next section. 
	Let $A$ be a subset of $\bbn$.
	We say that $A$ is 
	\begin{enumerate}
		\item an \emph{IP-set} if there exists a sequence $\{p_n\}_{n=1}^\infty$ in $\bbn$  such that for any non-empty finite subset $F$ of $\bbn$, $\sum_{n\in F} p_n\in A$;
		\item an \emph{AP-set} if it contains arbitrarily long arithmetic progressions,
		that is, for every $m\in\bbn$ there exist $a,b\in\bbn$ such that $a,a+b,\dotsc,a+(m-1)b\in A$;
		\item a \emph{J-set} if for every finite collection $\big\{\{p_n^i\}_{n=1}^\infty\colon i=1,2,\dotsc,m\big\}$ of sequences in $\bbn$ there exist $r\in\bbn$ and a non-empty finite subset $F$ of $\bbn$
		such that $r+\sum_{n\in F}p_n^i \in A $ for $i=1,2,\dotsc,m$;
		\item \emph{piecewise syndetic} if there exists $\ell>0$ such that $\bigcup_{i=0}^\ell (A-i)$ contains
		arbitrarily long integer intervals. That is, for any $m\in\bbn$ there exists $n\in\bbn$ 
		with $n,n+1,\dotsc,n+m\in \bigcup_{i=0}^\ell (A-i)$.
	\end{enumerate}
	The \emph{upper density} of $A$ is defined by
	\[
	d^*(A)=\limsup_{n\to\infty}\frac{|A\cap [1,n]|}{n},
	\]
	where $|\cdot |$ denotes the cardinality of a set,
	and the \emph{upper Banach  density} of $A$ is  defined by
	\[
	BD^*(A)=\limsup_{(n-m)\to\infty}\frac{|A\cap [m,n]|}{n-m+1}.
	\]
	
	Let $\calf_{\textrm{inf}}$ (resp. $\calf_{\textrm{ap}}$,
	$\calf_{\textrm{J}}$, $\calf_{\textrm{ps}}$, $\calf_{\textrm{pud}}$, $\calf_{\textrm{pubd}}$) be the collection of all infinite subsets $\bbn$ 
	(resp. AP-sets, J-sets, piecewise syndetic sets, sets with positive upper density,
	sets with positive upper Banach density).
	It is clear that all those collections are Furstenberg families and translation $+$ invariant.
	It is easy to verify that $\calf_{\textrm{inf}}$, $\calf_{\textrm{pud}}$ and $\calf_{\textrm{pubd}}$
	have the Ramsey property. By the well-known van de Waerden Theorem, $\calf_{\textrm{ap}}$ has the Ramsey property. 
	It is also well known that $\calf_{\textrm{ps}}$ has the Ramsey property, see e.g. {\cite[Theorem 1.24]{F81}}.
	By {\cite[Theorem 2.14]{HS2010}}, $\calf_{\textrm{J}}$ has the Ramsey property.
	It is easy to verify that for $\calf = \calf_{\textrm{inf}}$, $\calf_{\textrm{pud}}$ or $\calf_{\textrm{pubd}}$ and for any $\alpha>0$ and $0<\gamma<1$, $A\in \calf$ implies
	$g_{\alpha,\gamma}(A)\in\calf$. The results for $\calf = \calf_{\textrm{ps}}$, $\calf_{\textrm{ap}}$ or $\calf_{\textrm{J}}$ were proved in Theorems 2.4, 2.5 and 4.6 of {\cite{Hindman2012}}, respectively.

	Now we introduce some basic facts about the Stone-\v{C}ech compactification $\beta\bbn$ of $\bbn$.
	We refer the reader to {\cite{HS11}} for more details.
	A Furstenberg family $\calf$ is called a \emph{filter} if for every $A_1,A_2\in\calf$, $A_1\cap A_2\in\calf$, and an \emph{ultrafilter} if it is a filter with the Ramsey property.
	For each $n\in\bbn$, it is easy to verify that $e(n)=\{A\subset\bbn\colon n\in A\}$
	is an ultrafilter. Let $\beta\bbn$ be the collection of all ultrafilters on $\bbn$.
	Endow $\bbn$ with the discrete topology.
	For every subset $A$ of $\bbn$, let $\widehat{A}=\{p\in\beta\bbn\colon A\in p\}$.
	The collection $\{\widehat{A}\colon A\subset\bbn\}$ forms a basis of a topology on $\beta\bbn$. With this topology, $\beta\bbn$ is compact and Hausdorff.
	Then the embedding $e\colon \bbn\to \beta\bbn$, $n\mapsto e(n)$, is a Stone-\v{C}ech compactification of $\bbn$.
	
	Since $(\bbn,+)$ is a semigroup, we extend the operation $+$ to $\beta\bbn$
	so that $(\beta\bbn,+)$ is a compact Hausdorff right topological semigroup, which has rich algebraic structures.
	An \emph{idempotent} $p\in\beta\bbn$ is an element satisfying $p+p=p$. A subset $L$ of $\beta\bbn$ is called a \emph{left ideal} of $\beta\bbn$ if $\beta\bbn + L\subset L$.
	A \emph{minimal left ideal} is a left ideal that does not
	contain any proper left ideal. 
	An idempotent in $\beta\bbn$ is called a \emph{minimal idempotent} if
	it is contained in some minimal left ideal of $\beta\bbn$. 
	A subset $A$ of $\bbn$ is an IP-set if and only if there exists an idempotent $p$ in $\beta\bbn$ with $A\in p$, see e.g.~ {\cite[Theorem 5.12]{HS11}}.
	In {\cite{BH90}} Bergelson and Hindman obtained a simpler characterization of central sets in terms of the algebra of $\beta\bbn$, that is,
	a subset $A$ of $\bbn$ is central if and only if there exists a 
	minimal idempotent $p$ in $\beta\bbn$ with $A\in p$. 
	
	In {\cite{HM1996}} Hindman, Maleki and Strauss introduced the concept
	of quasi-central sets algebraically, which was dynamically characterized by 
	Burns and Hindman in {\cite[Theorem 3.4]{BH2007}}. 
	In {\cite{BD2008}} Bergleson and Downarowicz introduced the notion of D-sets and obtained some dynamical characterizations of D-set. 
	The authors in {\cite{HS2009}} called C-sets the sets satisfying the conclusion of the strong Central Sets Theorem, and obtained an algebraic characterization of the C-sets.
	In {\cite{Li2012}} the first author of this paper obtained a dynamical characterization of C-sets in $\bbn$. See also {\cite{J2018}} for a semigroup version.
	Here we uniformly define these sets algebraically. 
	A subset $A$ of $\bbn$ is called a quasi-central set (resp. D-set, C-set)
	if there exists an idempotent $p\subset \calf_{\textrm{ps}}$ 
	(resp. $p\subset \calf_{\textrm{pubd}}$ , $p\subset \calf_{\textrm{J}}$)
	with $A\in p$.

\section{Proof of the main results}
	
	The aim of this section is to prove Theorems~\ref{thm:main-result-1} and~\ref{thm:main-result-2}. 
	The key technique in the proofs is that we build a correspondence between the preservation of spectra and the lift
	property of suspension. 
	By the definition of suspension, we note that when we talk about the suspension of a dynamical system, the premise is that the dynamical system is invertible. 
	Thus first we will show that each essential $\calf$-set can be generated by the return time set of an invertible dynamical  system.

	\begin{lem}\label{lem:ess-f-set-onto}
		Let $\calf$ be a Furstenberg family.
		If $A$ is an essential $\calf$-set, 
		then there exist $X, T, x,y, U$ such that $T$ maps $X$ onto $X$ and $A$ is an essential $\calf$-set via $(X,T,x,y,U)$. 
	\end{lem}
	\begin{proof}
		By the definition, there exist a dynamical system $(Y, S)$, $a,b\in Y$ and a neighborhood $V$ of $b$ such that $a$ is $\calf$-strongly proximal to $b$ 
		and $N\big((a,b),V\times V\big)\subset A$. 
		Let $Z=\{0\}\cup \{\tfrac{1}{n}\colon n\in \bbn\}$ with the ordinary Euclidean metric. 
		Then $Z$ is a compact metric space. 
		Let $X=Z\times Y$ with product metric. 
		Then $X$ is a compact metric space. Let map $T:X\to X$, for $c\in Y$ and $n\in \bbn$, define $T(0,c)=(0,c)$, $T(1,c)=(1,S(c))$, and $T(\tfrac{1}{n+1},c)=(\tfrac{1}{n},c)$. 
		Then $T$ is continuous and  $(X,T)$ is a dynamical system. 
		Let $x=(1,a)$ and $y=(1,b)$, 
		and $U=\{1\}\times V$, then $U$ is a neighborhood of $y$. 
		Now we note that 
		for any neighborhood $U'$ of $y$ in $X$, 
		there exists a neighborhood $V'$ of $b$ such that $\{1\}\times V'\subset U'$, then
		\[
		N\big((x,y),U'\times U'\big)\supset N\big((x,y),(\{1\}\times V') \times (\{1\}\times V')\big)=N\big((a,b),V'\times V'\big)\in \calf. 
		\]
		Since $\calf$ is a Furstenberg family, 
		so $N((x,y),U'\times U')\in \calf$, 
		$x$ is $\calf$-strongly proximal to $y$. 
		And we also note that 
		\[
		N\big((x, y),U\times U\big)=N\big((a,b), V\times V\big)\subset A, 
		\]
		so
		$A$ is an essential $\calf$-set via $(X,T,x,y,U)$ and  $T$ maps $X$ onto $X$. 
	\end{proof}
	
	Now we introduce the natural extension of a dynamical system, which is an invertible dynamical system. 
	Let $(X,T)$ be a dynamical system with $T$ maps $X$ onto $X$. 
	Let 
	\[
	\widetilde{X}=\big\{(x_0,x_1,\dots)\in X^{\bbn_0}\colon Tx_{i+1}=x_i,i\in \bbn_0\big\}
	\] 
	be a subspace of $X^{\bbn_0}$ with product topology. 
	Then $\widetilde{X}$ is a nonempty closed subset of $X^{\bbn_0}$, so $\widetilde{X}$ is a compact metrizable space.  
	Define $\widetilde{T}:\widetilde{X}\to \widetilde{X}$ with 
	\[
	\widetilde{T} (x_0,x_1,\dots)=\big(T x_0, x_0,x_1,\dots\big) 
	\] 
	for any $(x_0,x_1,\dots)\in X$.
	Then $\widetilde{T}$ is a homeomorphism. 
	The dynamical system $(\widetilde{X},\widetilde{T})$ is called the \emph{natural extension} of $(X,T)$. 
	
	The following result regarding natural extension is well known and can be easily verified. 
	Here we also provide a proof for completeness. 
	
	\begin{lem}\label{lem,extensioin basis}
		Let $(X,T)$ be a dynamical system with $T$ maps $X$ onto $X$. 
		Let $(\widetilde{X},\widetilde{T})$ be the natural extension of $(X,T)$. 
		Then 
		\[
		\big\{\widetilde{X}\cap p_n^{-1}(U)\colon U \text{ is an open set of } X, n\in \bbn_0\big\}
		\]
		is a topological basis of $\widetilde{X}$, where $p_n$ is a projection mapping  from $X^{\bbn_0}$ to $X$ with $p_n(x_0,x_1,\dots)=x_n$. 
	\end{lem}
	\begin{proof}
		Let $x=(x_0,x_1,\dots)\in \widetilde{X}$ and $W$ be a neighborhood of $x$ in $\widetilde{X}$. 
		Then there exist $n\in \bbn_0$ and nonempty open subsets $W_0$, $W_1, \dotsc, W_n$ of $X$ such that 
		\[
		x\in (W_0\times W_1 \times \dotsb \times W_n \times X \times X\times\dotsb)\cap  \widetilde{X}  \subset W. 
		\]
		Then for any $i\in \{0,1,\dotsc, n\}$, $W_i$ is a neighborhood of $x_i$. 
		Let 
		\[
		U=\bigcap_{i=0}^n T^{-i} (W_{n-i}). 
		\]
		By the definition of $(\widetilde{X},\widetilde{T})$, for any $i\in \{0,1,\dotsc, n\}$, $T^i x_n=x_{n-i}\in W_{n-i}$, so we have $x_n\in U$.  
		Since $T$ is continuous, so $U$ is a neighborhood of $x_n$ in $X$, $\widetilde{X}\cap p_n^{-1}(U)$ is a neighborhood of $x$ in $\widetilde{X}$. 
		For any $y=(y_0,y_1,\dots)\in \widetilde{X}\cap p_n^{-1}(U)$, for any $i\in \{0,1,\dotsc, n\}$, we have 
		$y_{n-i}=T^i y_n\in T^i(U)\subset W_{n-i}$, that is, 
		\[
		y\in W_0\times W_1 \times \dotsb \times W_n \times X  \times X\times \dotsb. 
		\]
		Hence 
		\[
		x\in \widetilde{X}\cap p_n^{-1}(U) \subset  (W_0\times W_1 \times \dotsb \times W_n \times X  \times X\times \dotsb)\cap \widetilde{X} \subset W, 
		\]
		the proof is completed. 
	\end{proof}

	The following result reveal that in the definition of essential $\calf$-set we need only consider invertible dynamical systems. The technique is standard, see e.g. {\cite[Theorem 3.4] {Bergelson1996}}. Here we also provide a proof for completeness.
	
	\begin{lem}\label{lem:ess-f-set-homeo}
		Let $\calf$ be a translation $+$  invariant Furstenberg family.
		If $A$ is an essential $\calf$-set, 
		then there exist $X, T, x,y, U$ such that $T$ is a homeomorphism from $X$ onto $X$ and $A$ is an essential $\calf$-set via $(X,T,x,y,U)$. 
	\end{lem}
	\begin{proof} 
		By Lemma \ref{lem:ess-f-set-onto}, there exist $X, T, x,y, U$ such that $T$ maps $X$ onto $X$ and $A$ is an essential $\calf$-set via $(X,T,x,y,U)$. 
		Let $(\widetilde{X},\widetilde{T})$ be the natural extension of $(X,T)$. 
		By Lemma \ref{lem,extensioin basis}, $\{\widetilde{X}\cap p_n^{-1}(W)\colon W \text{ is an open set of } X, n\in \bbn_0\}$ is a topological basis of $\widetilde{X}$, where $p_n$ is a projection mapping  from $X^{\bbn_0}$ to $X$ with $p_n(x_0,x_1,\dots)=x_n$. 
		We have that $x$ is $\calf$-strongly proximal to $y$ by assumption, it implies that $y$ is a recurrent  point of $(X,T)$. 
		So there exist recurrent points $y_1,y_2,\dots$ in $(X,T)$ 
		such that 
		$
		(y,y_1,y_2,\dots)\in \widetilde{X}. 
		$
		Since $T$ is a surjection, there exist $x_1,x_2,\dots$ 
		such that 
		$
		(x,x_1,x_2,\dots)\in \widetilde{X}. 
		$
		
		Let  $a=(x,x_1,x_2,\dots)$, $b=(y,y_1,y_2,\dots)$ 
		and $V=\widetilde{X}\cap p_0^{-1}(U)$. 
		Then $V$ is a neighborhood of $b$ in $\widetilde{X}$ 
		and 
		\[
		N((a,b),V\times V)=N((x,y),U\times U)\subset A.
		\]
		For any neighborhood $V'$ of $b$ in $\widetilde{X}$, 
		there exist  $m\in \bbn_0$ and a neighborhood $U'$ of $y_m$   such that $\widetilde{X}\cap p_m^{-1}(U')\subset V'$. 
		Since $y_m$ is a recurrent point of $(X,T)$, there exists $n\in \bbn$ with $n>m$ such that $ T^n y_m=T^{n-m} y\in U'$. 
		So $T^{m-n} U'$ is a neighborhood of $y$. 
		For any $i\in N((x,y), T^{m-n} U'\times T^{m-n} U')$, we have $T^ix\in T^{m-n} U'$ and $T^iy\in T^{m-n} U'$, that is, $T^{i+n-m}x=T^{i+n}x_m\in U'$ and $ T^{i+n-m}y=T^{i+n} y_m\in U'$. 
		Thus 
		\[
		i+n\in N\Big((a,b), \big(\widetilde{X}\cap p_m^{-1}(U')\big)\times \big(\widetilde{X}\cap p_m^{-1}(U')\big)\Big)\subset N\big((a,b), V'\times V'\big). 
		\]
		So we have 
		\[
		n+N\big((x,y), T^{m-n} U'\times T^{m-n} U'\big)\subset  N\big((a,b),V'\times V'\big). 
		\]
		Since $\calf$ is a translation  $+$  invariant Furstenberg family. 
		So $N((a,b),V'\times V')\in \calf$, 
		so that 
		$a$ is $\calf$-strongly proximal to $b$. 
		Then $A$ is an essential $\calf$-set via $(\widetilde{X},\widetilde{T},a,b,V)$ and 
		$\widetilde{T}$ is a homeomorphism from $\widetilde{X}$ onto $\widetilde{X}$. 
	\end{proof}
	
	The following result shows that (2) implies (1) in Theorem \ref{thm:main-result-1}. 
	
	\begin{prop}\label{prop:prox-lift=>iter-inv}
		Let $\calf$ be a Furstenberg family and $\alpha>0$. 
		Assume that for any invertible dynamical system $(X,T)$ and any $x,y\in X$ with $x$ $\calf$-strongly proximal to $y$, there exists  $0<\gamma_0<1$ such that $(x,\gamma_0)$ is $\calf$-strongly proximal to $(y,\gamma_0)$ for the suspension $(Y,F_{\alpha^{-1}})$ of $(X,T)$. 
		If $A$ is an essential $\calf$-set,
		then for any $0<\gamma<1$,  $g_{\alpha,\gamma}(A)$ is also an essential $\calf$-set.
	\end{prop}
	
	\begin{proof} 
		Fix $0<\gamma<1$.
		Let $\delta=\min\{\gamma, 1-\gamma\}$ and 
		\[
		B=\big\{k\in \bbn\colon  \exists\, n\in A, \text{ s.t. }|k-n\alpha|<\delta\big\}.
		\]
		First we claim that $B\subset g_{\alpha,\gamma}(A)$.
		Indeed, for any $k\in B$, there exists $n\in A$ such that $|k-n\alpha|<\delta$.
		Then $\gamma-1\leq -\delta<k-n\alpha<\delta\leq\gamma$.
		This implies that $ k<n\alpha+\gamma<k+1$, so $\lfloor n\alpha+\gamma\rfloor =k$.
		
		As $A$ is an essential $\calf$-set, by Lemma~\ref{lem:ess-f-set-homeo}
		there exist an invertible dynamical system $(X,T)$, $x,y\in X$ and a neighborhood $U$ of $y$ such that  $x$ is $\calf$-strongly proximal to $y$ 
		and $N\big((x,y),U\times U\big)\subset A$. 
		By the assumption, there exists $0<\gamma_0<1$ such that $(x,\gamma_0)$ is $\calf$-strongly proximal to $(y,\gamma_0)$ for the suspension $(Y,F_{\alpha^{-1}})$. 
		Now let $\eps=\min\{\gamma_0,1-\gamma_0,\frac{\delta}{\alpha}\}$ and $W=U\times (\gamma_0-\eps, \gamma_0+\eps)$. 
		Then $W$ is a neighborhood of $(y,\gamma_0)$ and $
		N\Big(\big((x,\gamma_0),(y,\gamma_0)\big), W\times W\Big)$ is an essential $\calf$-set.
		It is sufficient to prove that
		\[
		N\Big(\big((x,\gamma_0),(y,\gamma_0)\big), W\times W\Big)\subset B.
		\]
		
		To this end, let 
		$k\in N\Big(\big((x,\gamma_0),(y,\gamma_0)\big), W\times W\Big)$.
		Then 
		\[
		F_{\alpha^{-1}}^k(x,\gamma_0)=(T^{\lfloor\tfrac{k}{\alpha}+\gamma_0\rfloor}x, \tfrac{k}{\alpha}+\gamma_0-\lfloor\tfrac{k}{\alpha}+\gamma_0\rfloor)\in W,
		\]
		and 
		\[
		F_{\alpha^{-1}}^k(y,\gamma_0)=(T^{\lfloor\frac{k}{\alpha}+\gamma_0\rfloor}y, \tfrac{k}{\alpha}+\gamma_0-\lfloor\tfrac{k}{\alpha}+\gamma_0\rfloor)\in W,
		\]
		that is,
		$T^{\lfloor\frac{k}{\alpha}+\gamma_0\rfloor}x, T^{\lfloor\frac{k}{\alpha}+\gamma_0\rfloor}y\in U$ and $\frac{k}{\alpha}+\gamma_0-\lfloor\frac{k}{\alpha}+\gamma_0\rfloor\in (\gamma_0-\eps, \gamma_0+\eps)$. 
		So $\lfloor\frac{k}{\alpha}+\gamma_0\rfloor\in N\big((x,y),U\times U\big)\subset A$ and
		$-\eps < \frac{k}{\alpha}-\lfloor\frac{k}{\alpha}+\gamma_0\rfloor <\eps$.
		Also
		$ -\delta\leq -\eps \alpha<k-\lfloor\frac{k}{\alpha}+\gamma_0\rfloor \alpha < \eps \alpha \leq \delta$, so
		$|k-\lfloor\frac{k}{\alpha}+\gamma_0\rfloor \alpha | < \delta$ and hence $k\in B$.
		This ends the proof.
	\end{proof}
	
	\begin{lem}\label{lem,A-A}
		Let $\calf$ be a Furstenberg family. 
		If $A$ is an essential $\calf$-set, then there exists an essential $\calf$-set $M\subset A$ such that for any  $m\in M$, $(-m+A)\cap A$ is  an  essential $\calf$-set.
		In particular $(A-A)\cap A$ is an essential $\calf$-set. 
	\end{lem}
	\begin{proof}
		Let $A$  be an essential $\calf$-set via $(X,T,x,y, U)$. 
		we may assume without loss of generality that $U$ is an open neighborhood of $y$. 
		Let $M=N\big((x,y),U\times U\big)$. Then $M$ is an essential $\calf$-set.
		Fix $m\in M$, that is, $T^mx,T^my\in U$. 
		Let $V=U\cap T^{-m}U$. Then $V$ is an open neighborhood of $y$ and $N\big((x,y),V\times V\big)$ is an essential $\calf$-set.
		For any $n\in N\big((x,y),V\times V\big)$, we have $T^n x,T^n y\in V\subset T^{-m}U$, then $T^{n+m} x,T^{n+m} y\in U$, so $n+m\in A$ and $n\in -m+A$. This shows that $N\big((x,y),V\times V\big)\subset -m+A$. 
		It is clear that $N\big((x,y),V\times V\big)\subset N\big((x,y),U\times U\big) \subset A$, so $N\big((x,y),V\times V\big)\subset (-m+A)\cap A$ and then $(-m+A)\cap A$ is an essential $\calf$-set. Finally, as $(-m+A)\cap A\subset (A-A)\cap A$,   $(A-A)\cap A$ is also an essential $\calf$-set. 
	\end{proof}
	
	We note that $\ess(\calf)\subset \calf$. 
	By the following  result, we have the implication (1) implies (2) in Theorem \ref{thm:main-result-1}.

	\begin{prop}\label{prop:iter-inv=>prox-lift}
		Let $\calf$ be a Furstenberg family and $\alpha>0$. 
		Assume that  $\ess(\calf)$ has the Ramsey property 
		and for any $A\in \ess(\calf)$ 
		and any $0<\gamma<1$,  $g_{\alpha,\gamma}(A)\in \calf$. 
		If $(X,T)$ is an invertible dynamical system and $x,y\in X$ with $x$ $\calf$-strongly proximal to $y$, 
		then for any $0<\gamma<1$, $(x,\gamma)$ is $\calf$-strongly proximal to $(y,\gamma)$ in the suspension $(Y,F_{\alpha^{-1}})$ of $(X,T)$. 
	\end{prop}
	
	\begin{proof} 
		Fix $0<\gamma<1$. 
		For any neighborhood $W$ of $(y,\gamma)$ in $Y$, 
		there exists a neighborhood $U$ of $y$ in $X$ and $0<\eps<\min\{\gamma,1-\gamma,\frac{1}{2\alpha}\}$ 
		such that $U\times (\gamma-\eps,\gamma+\eps)\subset W$. 
		Recall that for any $n\in \bbn$, 
		\[
		F_{\alpha^{-1}}^n(x,\gamma)=(T^{\lfloor\tfrac{n}{\alpha}+\gamma\rfloor}x,\tfrac{n}{\alpha}+\gamma-\lfloor\tfrac{n}{\alpha}+\gamma\rfloor).
		\]
		Let $A=N((x,y),U\times U)$ and 
		\[
		B=\{m\in A \colon \exists\, n\in \bbn \text{ s.t. } |\tfrac{n}{\alpha}-m|<\eps\}.
		\]
		As $x$ is $\calf$-strongly proximal to $y$, we have $A$ is an essential $\calf$-set. 
		We claim that $B$ is also an essential $\calf$-set.
		
		To this end, pick $L\in\bbn$ with $\frac{1}{L}<\eps$ and 
		let $\ell = L(\lfloor\tfrac{1}{\alpha}\rfloor+1)$.
		For each $i=0,1,\dotsc,\ell-1$, let
		\[
		A_i=\big\{m\in A\colon \exists\, n\in \bbn \text{ s.t. } 
		\tfrac{n}{\alpha}\in \big[m+\tfrac{i}{L}, m+\tfrac{i+1}{L}\big)\big\}.
		\]
		Then $A=\bigcup_{i=0}^{\ell-1}A_i$, 
		because for any $m\in A$ there exists $n\in \bbn$ such that $\tfrac{n}{\alpha}\in [m,m+\lfloor\tfrac{1}{\alpha}\rfloor+1)$. 
		As  $\ess(\calf)$  has the Ramsey property,
		there exists $j\in \{0,1,\dotsc,\ell-1\}$ such that $A_{j}\in \ess(\calf)$.  
		By Lemma~\ref{lem,A-A}, $(A_j-A_j)\cap A_j\in \ess(\calf)$. 
		Now is suffices to show that $(A_j-A_j)\cap A_j\subset B$. 
		Fix $m\in (A_j-A_j)\cap A_j$.
		There exist $m_1,m_2\in A_j$ with $m=m_1-m_2$.
		By the definition of the $A_j$, 
		there exist $n_1,n_2\in\bbn$ such that 
		\[
		\tfrac{n_1}{\alpha}\in [m_1+\tfrac{j}{L}, m_1+\tfrac{j+1}{L}) 
		\text{ and }
		\tfrac{n_2}{\alpha}\in [m_2+\tfrac{j}{L}, m_2+\tfrac{j+1}{L}). 
		\]
		Then 
		\[
		\tfrac{n_1-n_2}{\alpha}\in ((m_1-m_2)-\tfrac{1}{L}, (m_1-m_2)+\tfrac{1}{L}).
		\]
		Let $n=n_1-n_2$. Then 
		\[
		|\tfrac{n}{\alpha}-m|=|\tfrac{n_1-n_2}{\alpha}-(m_1-m_2)|< \tfrac{1}{L}< \eps. 
		\]
		This shows that $m\in B$.
		
		Now let 
		\[
		D=\{n\in\bbn \colon \exists\, m\in A  \text{ s.t. } |\tfrac{n}{\alpha}-m|<\eps\}.
		\]
		For any $m\in B\subset A$, there exists $n\in \bbn$ such that $|\frac{n}{\alpha}-m|<\eps$, 
		so we have  
		$n\in D$, 
		$-\eps < \tfrac{n}{\alpha}-m < \eps$ 
		and 
		$m\alpha-\eps \alpha< n < m\alpha+\eps \alpha$. 
		Since $\eps < \frac{1}{2\alpha}$, one has $\lfloor m\alpha+\eps \alpha\rfloor=n\in D$. Thus $g_{\alpha,\eps \alpha}(B)\subset D$.
		As $B$ is an essential $\calf$-set and by the assumption, 
		$g_{\alpha,\eps \alpha}(B)\in \calf$. 
		Since $\calf$ is a Furstenberg family, so $D\in \calf$. 
		
		For any $n\in D$, there exists $m\in A$ such that 
		$|\frac{n}{\alpha}-m|<\eps$, 
		so we have 
		$
		-\eps < \tfrac{n}{\alpha}-m < \eps$, 
		and 
		$
		\gamma-\eps< (\tfrac{n}{\alpha}+\gamma)-m < \gamma +\eps
		$. 
		Since $\eps<\min\{\gamma,1-\gamma\}$, one has $0<(\tfrac{n}{\alpha}+\gamma)-m <1$ and  $\lfloor\tfrac{n}{\alpha}+\gamma\rfloor=m \in A$. 
		Then 
		\begin{align*}
			F_{\alpha^{-1}}^n(x,\gamma)&=(T^{\lfloor\tfrac{n}{\alpha}+\gamma\rfloor} x, \tfrac{n}{\alpha}+\gamma-\lfloor\tfrac{n}{\alpha}+\gamma\rfloor)\\
			&=(T^m x, \tfrac{n}{\alpha}+\gamma-m)\\
			&\in U\times (\gamma-\eps,\gamma+\eps)\\
			&\subset W
		\end{align*}
		and 
		\[
		F_{\alpha^{-1}}^n(y,\gamma)=(T^m y, \tfrac{n}{\alpha}+\gamma-m)\in W. 
		\]
		This implies that $D\subset N\big(((x,\gamma),(y,\gamma)), W\times W\big)\in \calf$.  
		Thus $(x,\gamma)$ is $\calf$-strongly proximal to $(y,\gamma)$ in $(Y,F_{\alpha^{-1}})$. 
	\end{proof}
	
	\begin{proof}[Proof of Theorem~\ref{thm:main-result-1}] 
		It follows from Proposition~\ref{prop:prox-lift=>iter-inv} and  Proposition~\ref{prop:iter-inv=>prox-lift}. 
	\end{proof}

	The following theorem states that if $\calf$ is a translation $+$ invariant Furstenberg family with the Ramsey property, then the concept of essential $\calf$-set in this paper is consistent with that in {\cite{Li2012}}. 
	
	\begin{thm} \label{thm:ess-set-idempotent}
		Let $\calf$ be a translation $+$ invariant Furstenberg family with the Ramsey property. 
		If $A$ is a subset of $\bbn$, 
		then the following statements are equivalent: 
		\begin{enumerate}
			\item $A$  is an essential $\calf$-set. 
			\item there exists an idempotent $p\subset \calf$ with $A\in p$. 
			\item there exist a dynamical system $(X,T)$, a pair $x,y\in X$ where $x$ is $\calf$-strongly proximal to $y$, and a neighborhood $U$ of $y$ such that $A=N(x,U)$. 
			\item there is a decreasing sequence $\{C_n\}_{n=1}^\infty$ of subsets of $A$ such that $C_n\in \calf$ for any $n\in\bbn$, and for any $r\in C_n$ there exists $m\in \bbn$ such that $r+C_m\subset C_n$. 
		\end{enumerate}
	\end{thm}
	\begin{proof}
		Since $\calf$ is a translation $+$ invariant Furstenberg family with the Ramsey property. 
		By Lemma 3.4 in {\cite{Li2012}}, we have $h(\calf)=\{p\in \beta\bbn \colon p\subset \calf\}$ is a closed left ideal, so $h(\calf)$ is a subsemigroup of $\beta \bbn$. 
		
		That  (2) and (3) are equivalent follows from Theorem 4.11 in {\cite{Li2012}}.  
		
		That (3) and (4) are equivalent follows from Proposition 4.13 in {\cite{Li2012}}.  
		
		(3) implies (1). Note that $A=N(x,U)\supset N((x,y), U\times U)$. By the definition of essential $\calf$-set, we have  $A$ is an essential $\calf$-set. 
		
		(1) implies (4).  
		Assume that $A$ is an essential $\calf$-set via $(X,T,x,y, U)$. 
		Let $\{U_n\}_{n=1}^\infty$ be a neighborhood basis at $y$ such that for each $n\in \bbn$, $U_n\supset U_{n+1}$ is an open subset of $U$.  
		Let $C_n=N((x,y), U_n\times U_n)$, for any $n\in\bbn$. 
		Then  $\{C_n\}_{n=1}^\infty$ is a decreasing sequence of subsets of $A$.  
		For any $n\in\bbn$, $C_n\in \calf$. For any $r\in C_n$, we have $T^r y\in U_n$. Then $T^{-r}U_n$ is an open neighborhood of $y$, there exists $m\in \bbn$ such that $U_m\subset T^{-r} U_n $. 
		For any $k\in C_m$, we have $T^k x\in U_m\subset T^{-r}U_n$ and $T^k y\in U_m\subset T^{-r}U_n$. 
		So $T^{r+k}x\in U_n$ and $T^{r+k}y\in U_n$, that is, $r+k\in C_n$. 
		So we have $r+C_m\subset C_n$.   
	\end{proof}
	
	\begin{coro}\label{cor,Ramsey}
		If $\calf$ is a translation $+$ invariant Furstenberg family with the Ramsey property. 
		Then $\ess(\calf)$ has the Ramsey property. 
	\end{coro}
	\begin{proof} 
		By Theorem \ref{thm:ess-set-idempotent}, 
		for any subset $A$ of $\bbn$, $A$ is an essential $\calf$-set if and only if there exists an idempotent $p\subset \calf$ with $A\in p$. 
		Since each ultrafilter has the Ramsey property, so $\ess(\calf)$ has the Ramsey property. 
	\end{proof}

	\begin{proof}[Proof of Theorem~\ref{thm:main-result-2}] 
		Since $\calf$ is a translation $+$ invariant Furstenberg family with the Ramsey property, by Corollary \ref{cor,Ramsey}, $\ess(\calf)$ has the Ramsey property. 
		Since $\ess(\calf)\subset \calf$, 
		combining Proposition \ref{prop:iter-inv=>prox-lift}
		and Theorem \ref{thm:main-result-1}, 
		we have the statement (1) in Theorem \ref{thm:main-result-1}. 
		This proof is completed. 
	\end{proof}
	
	For $\calf = \calf_{\textrm{inf}}$, $\calf_{\textrm{ap}}$,
	$\calf_{\textrm{J}}$, $\calf_{\textrm{ps}}$, $\calf_{\textrm{pud}}$ or $\calf_{\textrm{pubd}}$, 
	we know that $\calf$ is a  translation $+$ invariant Furstenberg family with the Ramsey property. By Theorem \ref{thm:main-result-2}, we have the following consequence. 
	
	\begin{coro}\label{cor:ess-F-spectra-inv}
		For $\calf = \calf_{\textrm{inf}}$, $\calf_{\textrm{ap}}$,
		$\calf_{\textrm{J}}$, $\calf_{\textrm{ps}}$, $\calf_{\textrm{pud}}$ or $\calf_{\textrm{pubd}}$,
		and for any $\alpha>0$ and $0<\gamma<1$, $A\in \ess(\calf)$ implies
		$g_{\alpha,\gamma}(A)\in\ess(\calf)$. 
	\end{coro}

	\begin{rem}
		According to Theorem~\ref{thm:ess-set-idempotent}, for  a subset $A$ of $\bbn$,
		\begin{enumerate}
			\item $A$ is an essential $\calf_{\textrm{inf}}$-set if and only if it is an IP-set;
			\item $A$ is an essential $\calf_{\textrm{J}}$-set if and only if it is a C-set;
			\item $A$ is an essential $\calf_{\textrm{pubd}}$-set if and only if it is a D-set;
			\item $A$ is an essential $\calf_{\textrm{ps}}$-set if and only if it is a quasi-central set.
		\end{enumerate}
		Hence, Corollary~\ref{cor:IP-D-quasi-central-spectra} follows from Corollary~\ref{cor:ess-F-spectra-inv}.
	\end{rem}
	
	For a Furstenberg family $\calf$, the dual family of $\calf$, denoted by $\calf^*$,
	is defined as the collection $\{B\subset\bbn\colon A\cap B\neq\emptyset,\ \forall A\in\calf\}$.
	The dual family of $\ess(\calf)$ is denoted by $\ess^*(\calf)$.
	Using an elementary approach, Bergelson, Hindman and Kra proved the following result.
	\begin{thm}[{\cite[Theorem 2.5]{Bergelson1996}}] \label{thm:dual-prop}
		Let $\calf$ be a Furstenberg family with the Ramsey property.
		Assume that 
		\begin{enumerate}
			\item for each $A\in \calf$ there exist $x,y\in A$ with $x+y\in A$; 
			\item for any $\alpha>0$ and $0<\gamma<1$, $A\in\calf$ implies $g_{\alpha,\gamma}(A)\in\calf$. 
		\end{enumerate}
		Then for any $\alpha>0$ and $0<\gamma<1$, $A\in\calf^*$ implies $g_{\alpha,\gamma}(A)\in\calf^*$.
	\end{thm}
	
	By Lemma~\ref{lem,A-A}, the condition (1) of Theorem~\ref{thm:dual-prop} is fulfilled for any collection of essential $\calf$-sets. So we have the following consequence.
	
	\begin{coro}\label{cor:duall-ess-F-spectra-inv}
		For $\calf = \calf_{\textrm{inf}}$, $\calf_{\textrm{ap}}$,
		$\calf_{\textrm{J}}$, $\calf_{\textrm{ps}}$, $\calf_{\textrm{pud}}$ or $\calf_{\textrm{pubd}}$,
		and for any $\alpha>0$ and $0<\gamma<1$, $A\in \ess^*(\calf)$ implies
		$g_{\alpha,\gamma}(A)\in\ess^*(\calf)$. 
	\end{coro}
	

	\subsection*{Acknowledgements}
	The authors would like to thank Dr.\@ Leiye Xu for some helpful discussions. 
	We also express many thanks to the anonymous referee, whose comments have substantially improved this paper. 
	The authors were partially supported by NSF of China (12171298, 12222110)
	and NSF of Guangdong Province (2018B030306024).

\end{document}